\documentclass[11pt]{amsart}
\usepackage{amssymb,mathrsfs,graphicx,enumerate}

\title[   ]{Asymptotic analysis of Vlasov-type equations under strong local alignment regime}

\author[Kang]{Moon-Jin Kang}
\address[Moon-Jin Kang]{\newline Department of Mathematics, \newline The University of Texas at Austin, Austin, TX 78712, USA}
\email{moonjinkang@math.utexas.edu}

\author[Vasseur]{Alexis F. Vasseur}
\address[Alexis F. Vasseur]{\newline Department of Mathematics, \newline The University of Texas at Austin, Austin, TX 78712, USA}
\email{vasseur@math.utexas.edu}

\newtheorem{theorem}{Theorem}[section]
\newtheorem{lemma}{Lemma}[section]

\newtheorem{proposition}{Proposition}[section]
\newtheorem{remark}{Remark}[section]

\newcommand{\bbr}{\mathbb R}

\newcommand{\bbt} {\mathbb T}

\numberwithin{figure}{section}
\newcommand{\beq}{\begin{equation}}
\newcommand{\eeq}{\end{equation}}
\newcommand{\bsp}{\begin{split}}
\newcommand{\esp}{\end{split}}

\def\eps{\varepsilon }

\newcommand\adots{\mathinner{\mkern2mu\raise1pt\hbox{.}
\mkern3mu\raise4pt\hbox{.}\mkern1mu\raise7pt\hbox{.}}}

\def\charf {\mbox{{\text 1}\kern-.30em {\text l}}}

\begin{document}

\date{\today}

\subjclass{    } \keywords{}

\thanks{\textbf{Acknowledgment.} M.-J. Kang was supported by Basic Science Research Program through the National Research Foundation of Korea funded by the Ministry of Education, Science and Technology (NRF-2013R1A6A3A03020506). A. F. Vasseur was partially supported by the NSF Grant DMS 1209420.
}

\begin{abstract}
We consider the hydrodynamic limit of a collisionless and non-diffusive kinetic equation under strong local alignment regime. The local alignment is first considered by Karper, Mellet and Trivisa in \cite{K-M-T-2}, as a singular limit of an alignment force proposed by Motsch and Tadmor in \cite{M-T-1}. As the local alignment strongly dominate, a weak solution to the kinetic equation under consideration converges to the local equilibrium, which has the form of mono-kinetic distribution. We use the relative entropy method and weak compactness to rigorously justify the weak convergence of our kinetic equation to the pressureless Euler system. 
\end{abstract}
\maketitle \centerline{\date}


\section{Introduction}
\setcounter{equation}{0}
Recently, a variety of mathematical models capturing the emergent phenomena of such as a flock of birds, a swarm of bacteria or a school of fish have received lots of attention extensively in the mathematical community as well as the physics, biology, engineering and social science, etc. (See for instance \cite{C-C-R,C-K-F-L,M-T-2,Z-E-P} and the refereces therein.) In particular, the flocking model introduced by Cucker and Smale in \cite{C-S} has received a considerable attention. (See for instance \cite{C-F-R-T,H-L,H-T,Sh}.) In \cite{M-T-1}, Motsch and Tadmor improved the Cucker-Smale model by considering new interaction, which is non-local and non-symmetric alignment. Recently, in \cite{K-M-T-3}, Karper, Mellet and Trivisa introduced a strong local alignment interaction as the singular limit of the alignment proposed by Motsch-Tadmor. 

In this paper, we consider the Vlasov-type equation with the local alignment interaction, which is described by
\begin{align}
\begin{aligned} \label{main-eq}
\partial_t f + v\cdot\nabla_x f  -\lambda \nabla_v\cdot(v f) +  \nabla_v\cdot ((u-v)f) =0.
\end{aligned}
\end{align} 
Here, $f=f(t,x,v)$ is the one-particle distribution function at position $x\in\bbr^d$ with velocity $v\in\bbr^d$ at time $t>0$, and $\lambda\ge 0$ is a frictional coefficient.  In the force term $\nabla_v\cdot ((u-v)f)$, $u$ denotes the averaged local velocity defined by
\[
u=\frac{\int_{\bbr^d} v f  dv}{\int_{\bbr^d} f dv},
\]
thus the force term governs the local alignment interaction.\\
In \cite{K-M-T-1,K-M-T-3}, the authors studied the kinetic Cucker-Smale model with the local alignment term in \eqref{main-eq}. The local alignment force can be regarded as the singular limit of the following kinetic Motsch-Tadmor model \cite{M-T-1}:
\begin{align}
\begin{aligned} \label{M-T}
&\partial_t f + v\cdot\nabla_x f+\nabla_v\cdot (L[f] f) =0,\\
&L[f](x,v) = \frac{\int_{\bbr^{2d}} K^{r}(x-y) f(y,w)(w-v)dwdy}{\int_{\bbr^{2d}} K^{r}(x-y) f(y,w)dwdy}.
\end{aligned}
\end{align}
Here, the kernel $K^r$ is a communication weight and $r$ denotes the radius of influence of $K^r$. We know that the non-local alignment $L[f]$ is not symmetric interaction.\\
In \cite{K-M-T-3}, the authors considered the case that the communication weight is extremely concentrated nearby each agent. That is, they rigorously justified the singular limit $r\rightarrow 0$, i.e., as $K^r$ converges to the Dirac distribution $\delta_0$, the Motsch-Tadmor kernel $K^r$ converges to a local alignment term:
\[
L[f](x,v) \rightarrow \frac{\int_{\bbr^{d}} f(x,w)(w-v)dw}{\int_{\bbr^{d}} f(x,w)dw}=:u-v.
\]


We here address a natural question on the asymptotic regime corresponding to a strong local alignment force in the revised Motsch-Tadmor model \eqref{main-eq}. More precisely, we aim to rigorously justify the hydrodynamic limit of the singular equation:
\begin{align}
\begin{aligned} \label{vlasov}
&\partial_t f^{\eps} + v\cdot\nabla_x f^{\eps} -\lambda \nabla_v\cdot(v f^{\eps} ) +  \frac{1}{\eps}\nabla_v\cdot ((u^{\eps}-v)f^{\eps}) =0 ,\quad (x,v)\in \bbr^d\times\bbr^d,\quad t>0,\\
& f^{\eps}(0,x,v) = f_0^{\eps}(x,v),
\end{aligned}
\end{align} 
where $\displaystyle u^{\eps}=\frac{\int_{\bbr^d} v f^{\eps}  dv}{\int_{\bbr^d}  f^{\eps} dv}$ is the averaged local velocity.\\ 
The analytical difficulty for such issue is mainly due to the singularity of the local equilibrium. More precisely, the local equilibrium for \eqref{vlasov} has the form of mono-kinetic distribution as
\beq\label{mono-ki}
\rho(x,t) \delta(v-u(x,t)).
\eeq
For this reason, we will consider the hydrodynamic limit in the sense of distributions. Under this hydrodynamic regime, the associated limit system is the well-known damped pressureless Euler system:
\begin{align}
\begin{aligned} \label{PE}
&\partial_t \rho + \nabla\cdot (\rho u) = 0,  \\
&\partial_t (\rho u) + \nabla \cdot (\rho u\otimes u) = -\lambda \rho u , \quad x \in \bbr^d,~t > 0,\\
&\rho|_{t=0} = \rho_0,\quad u |_{t=0} = u_0.
\end{aligned}
\end{align}
When there is no friction, i.e., $\lambda= 0$, the system \eqref{PE} simply becomes the pressureless Euler system, which serve as a mathematical model for the formation of large-scale structures in astrophysics and the aggregation of sticky particles \cite{S-S-Z,Ze}. We refer to \cite{B-G,Bo,H-W,W-R-S} as the study on the existence and uniqueness of weak solutions to the pressureless Euler system in one space dimention, and the $\delta$-shock formation was studied in \cite{C-L}. Concerning the study on the multidimensional pressureless Euler system with non-local alignment force capturing flocking behavior, we refer to \cite{H-K-K1,M-T-1}. 
The macroscopic models studied in \cite{H-K-K1,H-K-K2,M-T-1} was formally derived from kinetic Cucker-Smale model \cite{H-T} and Motsch-Tadmor model \cite{M-T-1} under the mono-kinetic ansatz \eqref{mono-ki}, respectively.

If we consider the stochastic particles as the Brownian motion in \eqref{vlasov}, the kinetic equation \eqref{vlasov} becomes Vlasov-Fokker-Planck type equation, in which the Laplacian term with respect to $v$ is added as follows.
\begin{align}
\begin{aligned} \label{model-1}
\partial_t f^{\eps} + v\cdot\nabla_x f^{\eps} -\lambda \nabla_v\cdot(v f^{\eps} ) +  \frac{1}{\eps}\nabla_v\cdot ((u^{\eps}-v)f^{\eps}) -\frac{1}{\eps}\Delta_v f^{\eps} =0.
\end{aligned}
\end{align}
This type model have been studied in \cite{C-C-K,K-M-T-1,K-M-T-2,K-M-T-3}. In \cite{K-M-T-2}, the authors justified the hydrodynamic limit of \eqref{model-1} with the non-local alignment term due to Cucker-Smale model, under the strong local alignment and strong noise regime. As a mathematical tool, they used the relative entropy method with strictly convex entropy to show the strong convergence from \eqref{model-1} to the isothermal Euler system with non-local alignment. Indeed, the solution $f$ of \eqref{model-1} strongly converges to the local Maxwellian $\rho\exp(-\frac{|v-u|^2}{2})$ as $\eps\rightarrow 0$, thanks to the existence of strictly convex entropy for the isothermal Euler system.\\ 
On the other hand, for the pressureless Euler system \eqref{PE}, we have a convex entropy given by 
\beq\label{entropy-11}
\mathcal{E}(\rho, u)=\rho\frac{|u|^2}{2},
\eeq
which is not strictly convex with respect to $\rho$. Notice that \eqref{entropy-11} is physically the kinetic energy, but regarded as an entropy in the theory of conservation laws. We will use the relative entropy method with \eqref{entropy-11} in the weak sense, compared to the previous results based on the strictly convex entropy. (See for example \cite{C-C-K,K-M-T-2,M-V}.)\\

Our main tool based on the relative entropy method follows a program initiated in \cite{L,L-V,S-V,V, V-1} for the study on the stability of inviscid shocks for the scalar or system of conservation laws verifying a certain entropy condition. That has been used as an important tool in the study of asymptotic limits to conservation laws as well. (See for instance \cite{B-G-L-1,B-G-L-2,B-T-V,B-V,C-V,G-J-V-2,G-S,K-V,L-M,M-S,M-V,S,Y}.)\\

The rest of this paper is organized as follows. In Section 2, we first present the existence results of the kinetic equation \eqref{vlasov} and the asymptotic system \eqref{PE}, then the main theorem of this article. In Section 3 is devoted to the proof of Theorem \ref{main thm}.   

\section{Preliminaries}
\setcounter{equation}{0}
In this section, we present our main result on the hydrodynamic limit from a weak solution $f^{\eps}$ of the kinetic equation\eqref{vlasov} to a strong solution $(\rho, u)$ of the asymptotic system \eqref{PE}. For that, we first need to present the existence result for the weak solution of \eqref{vlasov} and the classical solution of \eqref{PE}. 
\subsection{Existence of weak solutions to \eqref{vlasov}} 
We here say that $f^{\eps}$ is a weak solution of \eqref{vlasov} if for any $T>0$,
\begin{align}
\begin{aligned} \label{weak-reg}
& f^{\eps}\in C(0,T; L^1(\bbr^{2d}))\cap L^{\infty}((0,T)\times\bbr^{2d}),\\
& |v|^2 f \in L^{\infty}(0,T; L^1(\bbr^{2d})),
\end{aligned}
\end{align}
and \eqref{vlasov} holds in the sense of distribution, that is, for any $\psi\in C_c^{\infty} ([0,T)\times\bbr^{2d})$, the weak formulation holds
\begin{align}
\begin{aligned} \label{weak-form}
&   \int_0^t\int_{\bbr^{2d}} f^{\eps} [ \partial_t\psi + v\cdot\nabla_x\psi -\lambda v\cdot\nabla_v\psi+\frac{1}{\eps}(u^{\eps}-v) \cdot\nabla_v\psi ]dvdxds\\
&\qquad +\int_{\bbr^{2d}} f^{\eps}_0\psi(0,\cdot) dvdx =0.
\end{aligned}
\end{align}
Before stating the existence of weak solution to \eqref{vlasov}, we define a kinetic entropy $\mathcal{F}(f^{\eps})$ mathematically and dissipation $ \mathcal{D}_{\eps}(f^{\eps})$ for \eqref{vlasov} by
\begin{align*}
\begin{aligned} 
&\mathcal{F}(f^{\eps}):=\int_{\bbr^d} \frac{|v|^2}{2} f^{\eps} dv,\\
& \mathcal{D}_{\eps}(f^{\eps}) : = \int_{\bbr^{2d}} f^{\eps} |u^{\eps}-v|^2 dx dv .
\end{aligned}
\end{align*}
\begin{proposition}\label{prop-weak} 
For any $\eps>0$, assume that $f^{\eps}_0$ satisfies 
\beq\label{weak-initial}
f^{\eps}_0\ge0,\quad f^{\eps}_0 \in L^{1}\cap L^{\infty}(\bbr^{2d}),\quad |v|^2 f^{\eps}_0 \in L^{1}(\bbr^{2d}).
\eeq
Then there exists a weak solution $f^{\eps}\ge 0$ of \eqref{vlasov} in the sense of \eqref{weak-reg} and \eqref{weak-form}. Moreover, $f^{\eps}$ satisfies the mass conservation $\|f^{\eps}(t)\|_{L^{1}(\bbr^{2d})}=\|f^{\eps}_0\|_{L^{1}(\bbr^{2d})}$ for all $t>0$, and the entropy inequality
\begin{align}
\begin{aligned}\label{kinetic-ineq}
&\int_{\bbr^d} \mathcal{F}(f^{\eps})(t) dx + \frac{1}{\eps} \int_0^t  \mathcal{D}_{\eps}(f^{\eps})(s) ds +\lambda \int_0^t \int_{\bbr^{2d}}  |v|^2 f^{\eps} dvdxds  \le \int_{\bbr^d} \mathcal{F}(f^{\eps}_0) dx.
\end{aligned}
\end{align}
\end{proposition}
We make use of the main result in \cite{K-M-T-1} to prove this result. In \cite{K-M-T-1}, the authors showed the existence of weak solutions to the kinetic Cucker-Smale model with local alignment, noise, self-propulsion and friction:
\begin{align}
\begin{aligned} \label{KMT}
&\partial_t f + v\cdot\nabla_x f +\nabla_v\cdot (L[f] f)+ \nabla_v\cdot ((u-v)f)\\
&\qquad = \mu \Delta_v f -  \nabla_v\cdot ((a-b|v|^2)vf) ,\\
&L[f] = \int_{\bbr^{2d}} \psi(x-y) f(y,w) (w-v) dw dy, \quad (x,v)\in \bbr^d\times\bbr^d,\quad t>0,
\end{aligned}
\end{align}
where the kernel $\psi$ is a smooth symmetric, $\mu, a,b$ are nonnegative constants.\\
From their analysis in \cite{K-M-T-1}, we notice that all terms except the local alignment term $\nabla_v\cdot ((u-v)f)$ are not crucial in the existence theory of weak solutions to \eqref{KMT}. Therefore the same existence theory can be applied to our case \eqref{vlasov}, so we refer \cite{K-M-T-1} for its proof.\\
On the other hand, we notice that the entropy inequality \eqref{kinetic-ineq} for \eqref{vlasov} is simply reduced more than that for \eqref{KMT} with $\mu>0$. Indeed, at least formally, multiplying \eqref{vlasov} by $|v|^2$ and integrating this over the phase space $\bbr^{2d}$, we have
\begin{align*}
\begin{aligned}
&   \frac{d}{dt}\int_{\bbr^{2d}} \frac{|v|^2}{2}  f^{\eps} dvdx + \lambda \int_{\bbr^{2d}} |v|^2 f^{\eps} dvdx
+ \frac{1}{\eps}\int_{\bbr^{2d}} f^{\eps} |u^{\eps}-v|^2 dx dv=0.
\end{aligned}
\end{align*}
This justifies the inequality \eqref{kinetic-ineq} rigorously by the standard method in \cite{K-M-T-1}.

\subsection{Existence of classical solutions to \eqref{PE}} 
We here present the local existence of classical solution to the pressureless Euler system \eqref{PE} as the following Proposition.  

\begin{proposition}\label{prop-strong} 
Assume that 
\beq\label{strong-initial}
\rho_0> 0\quad \mbox{in}~ \bbr^d\quad\mbox{and}\quad (\rho_0, u_0)\in H^{s}(\bbr^d)\times H^{s+1}(\bbr^d)\quad\mbox{for}~s>\frac{d}{2}+1.
\eeq
Then, there exists $T_*>0$ such that \eqref{PE} has a unique classical solution $(\rho, u)\in \mathcal{X}$ where $\mathcal{X}$ is the solution space defined by 
\begin{align*}
\begin{aligned}
\mathcal{X}&:= \{(\rho, u) : \rho\in C^0([0,T_*];H^{s}(\bbr^d))\cap C^1([0,T_*];H^{s-1}(\bbr^d)),\\
&\hspace{2cm} u\in C^0([0,T_*];H^{s+1}(\bbr^d))\cap C^1([0,T_*];H^{s}(\bbr^d))\}.
\end{aligned}
\end{align*}
Moreover, the smooth solution $(\rho, u)$ satisfies the entropy equality
\[
\frac{d}{dt}\mathcal{E}(\rho, u) +2\lambda \mathcal{E}(\rho, u) =0,
\]
where $\mathcal{E}(\rho, u)$ is an entropy for \eqref{PE} given by \eqref{entropy-11}.
\end{proposition}
\begin{remark}
For $s>\frac{d}{2}+1$, by Sobolev inequality, the solution $(\rho, u)$ obtained from the theorem above is indeed a classical solution, i.e, $(\rho, u)\in {\mathcal C}^1(\bbr^d \times [0,T_*])$.
\end{remark}
Proposition \ref{prop-strong} is the well-known result. We refer to \cite{H-K-K1} and \cite{H-K-K2} for its proof. In particular, the pressureless Euler system \eqref{PE} with damping $\lambda>0$ has a global-in-time classical solution, provided the initial data is small enough.

\subsection{Main result} 
We here state our result on the hydrodynamic limit of \eqref{vlasov}. For any fixed initial data $(\rho_0,u_0)$ satisfying \eqref{strong-initial}, we will consider a proper initial data $f^{\eps}_0$ satisfying \eqref{weak-initial} and
\begin{itemize}
\item
(${\mathcal A}1$): $\int_{\bbt^d} (\mathcal{F}( f^{\eps}_0) -  \mathcal{E}(\rho_0, u_0) ) dx =\mathcal{O}( \sqrt{\eps}),$
\vspace{0.2cm}
\item
(${\mathcal A}2$): $\rho^{\eps}_0 - \rho_0 =\mathcal{O}(\sqrt{\eps})$ \quad \mbox{in}~ $L^{\infty}(\bbr^d)$,
\vspace{0.2cm}
\item
(${\mathcal A}3$): $ u^{\eps}_0 - u_0 =\mathcal{O}(\sqrt{\eps})$\quad \mbox{in}~ $L^{\infty}(\bbr^d)$,
\end{itemize}
where $\rho^{\eps}_0$ and $u^{\eps}_0$ are defined by
\[
\rho^{\eps}_0=\int_{\bbr^d} f^{\eps}_0 dv,\quad u^{\eps} = \frac{1}{\rho^{\eps}_0}\int_{\bbr^d} v f^{\eps}_0 dv.
\]

\begin{theorem}\label{main thm} 
Assume that the initial datas $f^{\eps}_0$ and $(\rho_0, u_0)$ satisfy \eqref{weak-initial}, \eqref{strong-initial} and $({\mathcal A}1)$-$({\mathcal A}3)$.
Let $f^{\eps}$ be the weak solution to \eqref{vlasov} satisfying \eqref{kinetic-ineq}, and $(\rho, u)$ be the classical solution to \eqref{PE}. Then, we have
\begin{align}
\begin{aligned}\label{main-ineq}
\int_{\bbr^d}\rho^{\eps} (t) |(u^{\eps} - u)(t)|^2 dx \le C(T_*)\sqrt{\eps},\quad t\le T_*,
\end{aligned}
\end{align}
where $\rho^{\eps}=\int_{\bbr^d} f^{\eps} dv$, $\rho^{\eps}u^{\eps} = \int_{\bbr^d} v f^{\eps} dv$ and the constant $C(T_*)$ depends on $T_*$.\\
Therefore, we have the following convergences
\begin{align}
\begin{aligned}\label{w-conv}
&\rho^{\eps}\rightharpoonup \rho\quad\mbox{weakly} \quad \mbox{in} ~\mathcal{M}([0,T_*]\times\bbr^d),\\
&\rho^{\eps}u^{\eps} \rightharpoonup {\rho}u\quad\mbox{weakly} \quad \mbox{in} ~\mathcal{M}([0,T_*]\times\bbr^d),\\
&\rho^{\eps}u^{\eps}\otimes u^{\eps} \rightharpoonup {\rho}u\otimes u \quad\mbox{weakly} \quad \mbox{in} ~\mathcal{M}([0,T_*]\times\bbr^d),\\
&\int_{\bbr^d}f^{\eps} |v|^2 dv \rightharpoonup \rho |u|^2 \quad\mbox{weakly} \quad \mbox{in} ~\mathcal{M}([0,T_*]\times\bbr^d),
\end{aligned}
\end{align}
where $\mathcal{M}([0,T_*]\times\bbr^d)$ is the space of nonnegative Radon measures on $[0,T_*]\times\bbr^d$. Moreover, for any $\psi\in C^1(\bbr^d)$ with $\nabla\psi\in L^{\infty}(\bbr^d)$,
\beq\label{dirac-con}
\int_{\bbr^d}f^{\eps} \psi(v) dv \rightharpoonup \rho \psi(u) \quad\mbox{weakly} \quad \mbox{in} ~\mathcal{M}([0,T_*]\times\bbr^d).
\eeq
\end{theorem}

\section{Proof of Theorem \ref{main thm}} 
In this section, we present the proof of the Theorem \ref{main thm}. We first use the relative entropy method
to derive a relative entropy inequality \eqref{main-ineq}, which underlies the proof of the convergence \eqref{w-conv}. 
To use the relative entropy method, we consider the entropy $\mathcal{E}(\rho, u)=\rho\frac{|u|^2}{2}$, as mentioned in Introduction. Since $\mathcal{E}(\rho, u)$ is not strictly convex with respect to the density $\rho$, we will not be able to get the strong convergence of $\rho$ by the relative entropy method with the entropy \eqref{entropy-11}. Thus, after obtaining \eqref{main-ineq}, we will use the mass conservation law and \eqref{main-ineq} to get the weak convergence of $\rho$.  
\subsection{Relative entropy estimates} In this part, we show the inequality \eqref{main-ineq}.\\
We begin by defining the following notations: 
\[
P=\rho u,\quad U={\rho \choose P},\quad A(U)={ P^T \choose \frac{P\otimes P}{\rho}},\quad F(U)={0 \choose - P},
\]
Then, we can rewrite \eqref{PE} as a  simpler form, that is the system of conservation laws of the form
\beq\label{system}
\partial_t U + \mbox{div}_x A(U) = \lambda F(U).
\eeq
By the theory of conservation laws, the system \eqref{system} has a convex entropy $\mathcal{E}(\rho, u)=\rho\frac{|u|^2}{2}$ with entropy flux $G$ given as a solution of
\beq\label{entropy-flux}
\partial_{U_i}  G_{j} (U) = \sum_{k=1}^{d+1}\partial_{U_k} \mathcal{E}(U) \partial_{U_i}  A_{kj} (U),\quad 1\le i\le d+1,\quad 1\le j\le d.
\eeq

 introduce the relative entropy and relative flux:
\begin{align}
\begin{aligned}\label{relative}
&\mathcal{E}(V|U)=\mathcal{E}(V)-\mathcal{E}(U)- D\mathcal{E}(U)\cdot(V-U),\\
&{A}(V|U)={A}(V)-{A}(U)- D{A}(U)\cdot(V-U),
\end{aligned}
\end{align}
where $D\mathcal{E}(U)$ and $D{A}(U)$ denote the gradients with respect to $U$, and $D{A}(U)\cdot(V-U)$ is a matrix defined as
\[
(D{A}(U)\cdot(V-U))_{ij} = \sum_{k=1}^{d+1} \partial_{U_k}{A}_{ij}(U)(V_k-U_k),\quad 1\le i\le d+1,\quad 1\le j\le d.
\]
Since $\mathcal{E}(U)=\frac{|P|^2}{2\rho}$, we have
\[
D\mathcal{E}(U) = { -\frac{|P|^2}{2\rho^2} \choose  \frac{P}{\rho}} = { -\frac{|u|^2}{2} \choose  u}.
\]
If we consider $V={q \choose qw}$, then we have
\begin{align}
\begin{aligned}\label{relative-e}
\mathcal{E}(V|U)&=\frac{q}{2}|w|^2 -\frac{\rho}{2}|u|^2 +\frac{|u|^2}{2} (q-\rho) - u (qw-\rho u)\\
&=\frac{q}{2}|u-w|^2.
\end{aligned}
\end{align}
The next proposition gives a key formulation to get the relative entropy inequality for the system of conservation laws \eqref{system}.  
\begin{proposition}\label{prop-key} 
Let $U$ be a strong solution of \eqref{system} and $\displaystyle V$ be a smooth function. Then, the following equality holds
\begin{align}
\begin{aligned} \label{key ineq}
\frac{d}{dt}\int_{\bbr^d}\mathcal{E}(V|U) dx & = \frac{d}{dt}\int_{\bbr^d}\mathcal{E}(V) dx-\int_{\bbr^d} \nabla_xD\mathcal{E}(U): A(V|U) dx\\
& \quad-\int_{\bbr^d} D\mathcal{E}(U) \cdot[\partial_tV + \mbox{div}_x A(V) -\lambda F(V)] dx\\
&\quad -\lambda \int_{\bbr^d} (D^2\mathcal{E}(U) F(U)\cdot (V-U) +  D\mathcal{E}(U) \cdot F(V) ) dx
\end{aligned}
\end{align}
\end{proposition}
The derivation for \eqref{key ineq} can be found in \cite{K-M-T-2,M-V}. We present its proof in Appendix for the reader's convenience.

We now apply \eqref{key ineq} to our issue together with the following macroscopic quantities corresponding to the weak solution $f^{\eps}$ obtained in previous section:
\[
P^{\eps}=\rho^{\eps}u^{\eps} ,\quad  U^{\eps}={\rho^{\eps} \choose P^{\eps}}
\]
where $\rho^{\eps}=\int_{\bbr^d} f^{\eps} dv$, $\rho^{\eps}u^{\eps} = \int_{\bbr^d} v f^{\eps} dv$.

\begin{lemma}\label{lem-ineq} 
Assume that the initial datas $U^{\eps}_0$ and $U_0$ satisfy \eqref{weak-initial}, \eqref{strong-initial} and $({\mathcal A}1)$-$({\mathcal A}3)$.
Let $U^{\eps}$ be a weak solution of \eqref{vlasov} satisfying \eqref{kinetic-ineq} and $U$ the classical solution of \eqref{PE}. Then, the following inequality holds
\begin{align}
\begin{aligned}\label{distance} 
&\int_{\bbr^d}\mathcal{E}(U^{\eps}|U)(t) dx\\
&\quad \le C\sqrt{\eps}+ C \int_0^t  \int_{\bbr^d} \mathcal{E}(U^{\eps}|U) dxds + \int_{\bbr^d}( \mathcal{F}(f^{\eps})(t)- \mathcal{F}(f^{\eps}_0))dx\\
& \qquad  - \int_0^t \int_{\bbr^d} D\mathcal{E}(U)  \cdot [\partial_t U^{\eps} + \mbox{div}_x A(U^{\eps}) -\lambda F(U^{\eps})] dxds \\
&\qquad -\lambda\int_0^t \int_{\bbr^d} (D^2\mathcal{E}(U) F(U)  \cdot(U^{\eps}-U) +  D\mathcal{E}(U) \cdot F(U^{\eps}) ) dx ds,
\end{aligned}
\end{align}
where the integrand $\partial_t U^{\eps} + \mbox{div}_x A(U^{\eps}) -\lambda F(U^{\eps})$ is regarded in the sense of distributions.
\end{lemma}
\begin{proof}
First of all, it follows from \eqref{key ineq} that 
\begin{align}
\begin{aligned} \label{ineq-1}
\int_{\bbr^d}\mathcal{E}(U^{\eps}|U)(t) dx&\le \int_{\bbr^d}\mathcal{E}(U^{\eps}_0|U_0) dx + \int_{\bbr^d}(\mathcal{E}(U^{\eps})(t) - \mathcal{E}(U^{\eps}_0))dx\\
&\quad   -\int_0^t\int_{\bbr^d} \nabla_xD\mathcal{E}(U): A(U^{\eps}|U) dx ds\\
&\quad -\int_0^t \int_{\bbr^d} D\mathcal{E}(U)  \cdot[\partial_t U^{\eps} + \mbox{div}_x A(U^{\eps}) -\lambda F(U^{\eps})] dxds\\
& \quad -\lambda \int_0^t \int_{\bbr^d} (D^2\mathcal{E}(U) F(U)  \cdot(U^{\eps}-U) +  D\mathcal{E}(U)  \cdot F(U^{\eps}) ) dx ds\\
&:= \sum_{k=1}^5 \mathcal{I}_k.
\end{aligned}
\end{align}
We estimate the three terms $\mathcal{I}_1$, $\mathcal{I}_2$ and $\mathcal{I}_3$ above as follows.\\
$\bullet$ {\bf Estimate of $\mathcal{I}_1$}
We use \eqref{relative-e}, assumption $({\mathcal A}3)$ and the mass conservation to get
\[
\mathcal{I}_1 = \frac{1}{2}\int_{\bbr^d}\rho^{\eps}_0|u^{\eps}_0 - u_0|^2 dx \le \mathcal{O}(\eps)\int_{\bbr^d}\rho^{\eps}_0dx \le C\eps.
\]
$\bullet$ {\bf Estimate of $\mathcal{I}_2$}
We decompose $\mathcal{I}_2$ as
\begin{align}
\begin{aligned} \label{I-2}
 \mathcal{I}_2 &=  \int_{\bbr^d}(\mathcal{E}(U^{\eps})(t) - \mathcal{F}(f^{\eps})(t))dx +  \int_{\bbr^d}( \mathcal{F}(f^{\eps})(t)- \mathcal{F}(f^{\eps}_0))dx\\
 &\quad + \int_{\bbr^d}( \mathcal{F}(f^{\eps}_0)- \mathcal{E}(U_0))dx +\int_{\bbr^d}( \mathcal{E}(U_0)- \mathcal{E}(U^{\eps}_0))dx\\
 &:=  \sum_{k=1}^4 \mathcal{I}_2^k.
\end{aligned}
\end{align}
We first notice that H$\ddot{\mbox{o}}$lder's inequality yields 
\begin{align}
\begin{aligned}\label{pre-mini}
\rho^{\eps} |u^{\eps}|^2 =\frac{|\rho^{\eps} u^{\eps}|^2}{\rho^{\eps}}=\frac{\Big|\int_{\bbr^d} v f^{\eps} dv \Big|^2}{\int_{\bbr^d}  f^{\eps} dv} \le \int_{\bbr^d} |v|^2 f^{\eps} dv.
\end{aligned}
\end{align}
Thanks to \eqref{pre-mini}, we have the minimization property 
\beq\label{mini}
\mathcal{E}(U^{\eps})\le \mathcal{F}(f^{\eps}), 
\eeq
which yields $\mathcal{I}_2^1 \le 0$.\\
The assumption $(\mathcal{A}1)$ directly gives $\mathcal{I}_2^3 = \mathcal{O}(\sqrt{\eps}) $.\\
We use the assumptions $(\mathcal{A}2)$ and $(\mathcal{A}3)$ to get
\[
\mathcal{I}_2^4 = \frac{1}{2} \int_{\bbr^d} \rho_0 (|u_0|^2 -|u_0^{\eps}|^2) + (\rho_0 - \rho_0^{\eps}) |u_0^{\eps}|^2 dx \le C\sqrt{\eps},
\]
where we have used $u_0^{\eps}=\mathcal{O}(1)$ due to $(\mathcal{A}2)$. Thus we have
\[
\mathcal{I}_2\le C\sqrt{\eps} + \int_{\bbr^d}( \mathcal{F}(f^{\eps})(t)- \mathcal{F}(f^{\eps}_0))dx
\]
$\bullet$ {\bf Estimate of $\mathcal{I}_3$} We first compute the relative flux as follows. Since
\[
A(U^{\eps}|U)=A(U^{\eps})-A(U)-DA(U)\cdot(U^{\eps}-U),
\]
and
\begin{align*}
\begin{aligned}
DA(U)\cdot(U^{\eps}-U) &= D_{\rho}A(U) (\rho^{\eps}-\rho) + D_{P_i} A(U) (P_i^{\eps} - P_i) \\
&={ ({P^{\eps}} -P)^T \choose -\frac{\rho^{\eps}-\rho}{\rho^2} P\otimes P +\frac{1}{\rho} P\otimes (P^{\eps} -P ) +   \frac{1}{\rho} (P^{\eps} -P )\otimes P },
\end{aligned}
\end{align*}
we have
\begin{align*}
\begin{aligned}
A(U^{\eps}|U)&={ 0 \choose  \frac{1}{\rho^{\eps}}P^{\eps}\otimes P^{\eps}- \frac{1}{\rho}P\otimes P +\frac{\rho^{\eps}-\rho}{\rho^2} P\otimes P -\frac{1}{\rho} P\otimes (P^{\eps} -P ) -\frac{1}{\rho} (P^{\eps} -P )\otimes P }\\
&= {0 \choose \rho^{\eps} (u^{\eps} -u)\otimes (u^{\eps} -u)}.
\end{aligned}
\end{align*}
Then, by using $D\mathcal{E}(U) = {-\frac{|u|^2}{2} \choose u}$, we have 
\begin{align*}
\begin{aligned}
\mathcal{I}_3 &= \int_0^t\int_{\bbr^d} \rho^{\eps} (u^{\eps} -u)\otimes (u^{\eps} -u) :  \nabla_x u dx ds \\
&\le C\|u\|_{L^{\infty}(0,T_*;W^{1,\infty}(\bbr^d))} \int_0^t\int_{\bbr^d} \rho^{\eps}|u^{\eps} -u|^2  dx ds.
\end{aligned}
\end{align*}
Therefore, by $\mathcal{E}(U^{\eps}|U)=\frac{\rho^{\eps}}{2} |u^{\eps} - u|^2$, we have
\[
\mathcal{I}_3 \le C \|u\|_{L^{\infty}(0,T_*;W^{1,\infty}(\bbr^d))} \int_0^t \int_{\bbr^d} \mathcal{E}(U^{\eps}|U) dxds.
\]
\end{proof}

To get the relative entropy inequality \eqref{main-ineq} from Lemma \ref{lem-ineq}, we need to estimate the last three terms in the right hand side of \eqref{distance} as follows.

\begin{lemma}\label{lem-remain} 
Under the same hypotheses as Lemma \ref{lem-ineq}, we have
\beq \label{consist}
\Big|\int_0^t\int_{\bbr^d} D\mathcal{E}(U)\cdot [\partial_tU^{\eps} + \mbox{div}_x A(U^{\eps}) -\lambda F(U^{\eps})] dx ds \Big| \le C\sqrt{\eps},
\eeq
and
\beq\label{last term}
\int_{\bbr^d}( \mathcal{F}(f^{\eps})(t)- \mathcal{F}(f^{\eps}_0))dx -\lambda\int_0^t \int_{\bbr^d} (D^2\mathcal{E}(U) F(U) (U^{\eps}-U) +  D\mathcal{E}(U) F(U^{\eps}) ) dx ds \le 0.
\eeq
\end{lemma}
\begin{proof}
$\bullet$ {\bf Estimate of \eqref{consist}} First of all, it follows from \eqref{weak-form} that
\begin{align*}
\begin{aligned} 
&\partial_t\rho^{\eps} + \nabla_x \cdot (\rho^{\eps} u^{\eps})=0,\\
&\partial_t(\rho^{\eps}u^{\eps}) + \nabla_x \cdot (\rho^{\eps} u^{\eps}\otimes u^{\eps}) +\lambda \rho^{\eps}u^{\eps} =\nabla_x\cdot \int_{\bbr^d}(u^{\eps}\otimes u^{\eps} - v\otimes v)f^{\eps}dv,\\
\end{aligned}
\end{align*}
in the distributional sense. This implies that
\begin{align*}
\begin{aligned} 
&\int_0^t\int_{\bbr^d} D\mathcal{E}(U) \cdot[\partial_tU^{\eps} + \mbox{div}_x A(U^{\eps}) -F(U^{\eps})] dx ds\\
&\hspace{2cm} = \int_0^t\int_{\bbr^d} D_P\mathcal{E}(U)\cdot \mbox{div}_x\Big(\int_{\bbr^d} (u^{\eps}\otimes u^{\eps} - v\otimes v)f^{\eps} dv \Big) dx ds\\
&\hspace{2cm} = -\int_0^t\int_{\bbr^d} \nabla_x D_P\mathcal{E}(U):\Big(\int_{\bbr^d} (u^{\eps}\otimes u^{\eps} - v\otimes v)f^{\eps} dv \Big)dx ds,
\end{aligned}
\end{align*}
Since $D_P\mathcal{E}(U) = u$ and $u^{\eps}\otimes u^{\eps} - v\otimes v=u^{\eps}\otimes (u^{\eps}-v) +(u^{\eps}- v)\otimes v$, we use H$\ddot{\mbox{o}}$lder's inequality to estimate
\begin{align*}
\begin{aligned} 
&\Big|\int_0^t\int_{\bbr^d} D\mathcal{E}(U) \cdot [\partial_tU^{\eps} + \mbox{div}_x A(U^{\eps})-F(U^{\eps})] dx ds \Big|\\
&\hspace{2cm} \le \int_0^t\int_{\bbr^d} |\nabla_x u| \Big|\int_{\bbr^d} (u^{\eps}\otimes u^{\eps} - v\otimes v)f^{\eps} dv \Big| dx ds\\
&\hspace{2cm} \le \|u\|_{L^{\infty}(0,T_*;W^{1,\infty})}\int_0^t\int_{\bbr^d} \Big|\int_{\bbr^d} (u^{\eps}\otimes (u^{\eps}-v) +(u^{\eps}- v)\otimes v)f^{\eps} dv \Big| dx ds\\
&\hspace{2cm} \le C\int_0^t \int_{\bbr^{2d}} (|u^{\eps}|+|v|) |u^{\eps}- v| f^{\eps} dv dx ds\\
&\hspace{2cm} \le C\int_0^t\Big( \underbrace{\int_{\bbr^{2d}} (|u^{\eps}|^2+ |v|^2) f^{\eps} dvdx}_{J} \Big)^{1/2} \Big( \int_{\bbr^{2d}} |u^{\eps}- v|^2 f^{\eps} dv dx\Big)^{1/2}  ds.
\end{aligned}
\end{align*}
By the minimization property \eqref{mini}, since
\[
J = \int_{\bbr^d} \rho^{\eps}|u^{\eps}|^2 dx +\int_{\bbr^{2d}} |v|^2 f^{\eps} dvdx \le 2\int_{\bbr^{2d}} |v|^2 f^{\eps} dvdx,
\]
the H$\ddot{\mbox{o}}$lder's inequality and entropy inequality \eqref{kinetic-ineq} yield
\begin{align*}
\begin{aligned} 
&\Big|\int_0^t\int_{\bbr^d} D\mathcal{E}(U) \cdot [\partial_tU^{\eps} + \mbox{div}_x A(U^{\eps})-F(U^{\eps})] dx ds \Big|\\
&\hspace{2cm} \le \Big( \int_0^t \int_{\bbr^d} \mathcal{F}(f^{\eps})(t) dx ds \Big)^{1/2}  \Big(  \int_0^t  \mathcal{D}_{\eps}( f^{\eps}) ds\Big)^{1/2}\\
&\hspace{2cm} \le \Big( t\int_{\bbr^d} \mathcal{F}(f^{\eps}_0) dx \Big)^{1/2}  \Big(  \int_0^t  \mathcal{D}_{\eps}( f^{\eps}) ds\Big)^{1/2}\\
&\hspace{2cm} \le C \sqrt{T_*}  \Big(  \int_0^t  \mathcal{D}_{\eps}( f^{\eps}) ds\Big)^{1/2} ,\quad t\le T_*,
\end{aligned}
\end{align*}
where $\int_{\bbr^d} \mathcal{F}(f^{\eps}_0) dx$ is bounded w.r.t. $\eps$ by the assumption $(\mathcal{A}1)$.\\
We use \eqref{kinetic-ineq} to estimate 
\beq\label{diss-zero}
\int_0^{t}  \mathcal{D}_{\eps}(f^{\eps})(s) ds \le \eps \int_{\bbr^d} \mathcal{F}(f^{\eps}_0) dx \le C\eps,
\eeq
which complete \eqref{consist}.\\
$\bullet$ {\bf Estimate of \eqref{last term}}
Since 
\[
D_{P}D_{\rho}\mathcal{E}(U) = - \frac{P}{\rho^2}= - \frac{u}{\rho} \quad \mbox{and}\quad  D_{P}D_{P}\mathcal{E}(U) =  \frac{1}{\rho}I,
\]
we have 
\begin{align*}
\begin{aligned} 
&-\lambda\int_0^t \int_{\bbr^d} (D^2\mathcal{E}(U) F(U) (U^{\eps}-U) +  D\mathcal{E}(U) F(U^{\eps}) ) dx ds \\
&\qquad = -\lambda\int_0^t \int_{\bbr^d} (|u|^2(\rho^{\eps}-\rho) -u(\rho^{\eps} u^{\eps}-\rho u) -\rho^{\eps} u u^{\eps}) dx ds \\
&\qquad = -\lambda\int_0^t \int_{\bbr^d} \rho^{\eps}| u -u^{\eps}|^2 dx ds +\lambda\int_0^t \int_{\bbr^d} \rho^{\eps}| u^{\eps}|^2 dx ds \\
&\qquad\le \lambda\int_0^t \int_{\bbr^d} \rho^{\eps}| u^{\eps}|^2 dx ds\\
&\qquad\le \lambda\int_0^t \int_{\bbr^{2d}} |v|^2 f^{\eps} dxdv ds\\
\end{aligned}
\end{align*}
where the last inequality is due to minimization property \eqref{pre-mini}.\\
Thus, we have
\begin{align*}
\begin{aligned} 
&\int_{\bbr^d}( \mathcal{F}(f^{\eps})(t)- \mathcal{F}(f^{\eps}_0))dx -\lambda\int_0^t \int_{\bbr^d} (D^2\mathcal{E}(U) F(U) (U^{\eps}-U) +  D\mathcal{E}(U) F(U^{\eps}) ) dx ds\\
&\qquad \le \int_{\bbr^d}( \mathcal{F}(f^{\eps})(t)- \mathcal{F}(f^{\eps}_0))dx +\lambda \int_0^t \int_{\bbr^d}\mathcal{F}(f^{\eps})(s) dxds.
\end{aligned}
\end{align*}
Since the kinetic entropy inequality \eqref{kinetic-ineq} yields
\[
\int_{\bbr^d} \mathcal{F}(f^{\eps})(t) dx  +\lambda\int_0^t \int_{\bbr^{2d}} |v|^2 f^{\eps} dxdv ds  \le \int_{\bbr^d} \mathcal{F}(f^{\eps}_0) dx,
\]
we have \eqref{last term}.
\end{proof}

By Lemma \ref{lem-ineq} and \ref{lem-remain}, we have the Gronwall type inequality
\[
\int_{\bbr^d}\mathcal{E}(U^{\eps}|U)(t) dx \le C\sqrt{\eps}+ C \int_0^t \int_{\bbr^d} \mathcal{E}(U^{\eps}|U) dxds.
\]
This implies
\[
\int_{\bbr^d}\mathcal{E}(U^{\eps}|U)(t) dx \le C\sqrt{\eps}.
\]
Since 
\beq
\mathcal{E}(U^{\eps}|U)(t) = \frac{1}{2}\rho^{\eps} (t) |(u^{\eps} - u)(t)|^2,
\eeq
we complete the relative entropy inequality \eqref{main-ineq}.

\subsection{Convergence} In this part, we use \eqref{main-ineq} to show the convergence \eqref{w-conv}, thus the convergence of the kinetic equation \eqref{vlasov} to the pressureless Euler system \eqref{system}.\\
\subsubsection{Convergence of $\rho^{\eps}$ and $\rho^{\eps}u^{\eps}$}
First of all, by Proposition \ref{prop-weak}, we have the conservation of mass 
\[
\|\rho^{\eps}(t)\|_{L^1(\bbr^d)} = \|\rho^{\eps}_0\|_{L^1(\bbr^d)},\quad t>0.
\]
Notice that by the assumption $(\mathcal{A}1)$, $\{\rho^{\eps}_0\}_{\eps>0}$ is uniformly bounded in $L^{\infty}(0,T_*;L^1(\bbr^d))$. \\
Thus  there exists $\tilde{\rho}$ such that
\beq\label{weak conv}
\rho^{\eps}\rightharpoonup  \tilde{\rho}\quad \mbox{weak-}* \quad \mbox{in} ~\mathcal{M}([0,T_*]\times\bbr^d),
\eeq
where $\mathcal{M}([0,T_*]\times\bbr^d)$ is the space of nonnegative Radon measures on $[0,T_*]\times\bbr^d$.\\
We now claim that 
\beq\label{ae}
\tilde{\rho}=\rho ~\mbox{on}~ [0,T_*]\times\bbr^d.
\eeq
We start with the fact that for any function $\phi\in C_c^{\infty}([0,T_*]\times\bbr^d)$, taking $\psi=\phi$ as a test function in \eqref{weak-form}, we find
\beq\label{eps-eq}
\int_0^t \int_{\bbr^d}(\rho^{\eps}\partial_t \phi + \rho^{\eps} u^{\eps} \cdot \nabla_x\phi)dxds + \int_{\bbr^d}\rho^{\eps}_0 \phi(0,x)dx=0.
\eeq
From \eqref{eps-eq}, we will derive the following equation of $\tilde{\rho}$:
\beq\label{claim-11}
\int_0^t \int_{\bbr^d}(\tilde{\rho}\partial_t \phi + \tilde{\rho} u \cdot \nabla_x\phi)dxds + \int_{\bbr^d}\rho_0 \phi(0,x)dx=0.
\eeq
The weak convergence \eqref{weak conv} yields
\beq\label{claim-1}
\int_0^t \int_{\bbr^d}\rho^{\eps}\partial_t \phi dxds \quad\rightarrow\quad \int_0^t \int_{\bbr^d}\tilde{\rho}\partial_t \phi dxds \quad\mbox{as}~ \eps\rightarrow 0.
\eeq
For the convergence of second term in \eqref{eps-eq}, we consider 
\begin{align}
\begin{aligned} \label{1flux}
&\int_0^t \int_{\bbr^d} (\rho^{\eps} u^{\eps} - \tilde{\rho} u)\cdot\nabla_x\phi dxds\\
&\qquad=\int_0^t \int_{\bbr^d} \rho^{\eps} (u^{\eps} -u) \cdot\nabla_x\phi  dxds + \int_0^t \int_{\bbr^d} (\rho^{\eps}- \tilde{\rho}) u\cdot\nabla_x\phi dxds\\
&\qquad=:\mathcal{I}_1 +\mathcal{I}_2.
\end{aligned}
\end{align}
We use the H$\ddot{\mbox{o}}$lder's inequality and the conservation of mass to estimate
\begin{align*}
\begin{aligned} 
\mathcal{I}_1&\le \|\nabla_x\phi\|_{\infty}\int_0^t \Big( \int_{\bbr^d} \rho^{\eps} dx\Big)^{1/2} \Big(\int_{\bbr^d}  \rho^{\eps} |u^{\eps} -u|^2 dx\Big)^{1/2}ds\\
&\le CT_* \|\rho^{\eps}_0\|_{L^1(\bbr^d)}^{1/2}\Big(\int_{\bbr^d} \rho^{\eps} |u^{\eps} -u|^2 dx\Big)^{1/2} .
\end{aligned}
\end{align*}
Then by the assumption $(\mathcal{A}2)$ and the relative entropy inequality \eqref{main-ineq}, we have
\[
\mathcal{I}_1\le C\eps^{1/4}\rightarrow 0 \quad\mbox{as}~ \eps\rightarrow 0.
\]
Since $u$ is smooth, \eqref{weak conv} yields 
\[
\mathcal{I}_2\rightarrow 0 \quad\mbox{as}~ \eps\rightarrow 0.
\]
Thus we have
\beq\label{claim-2}
\int_0^t \int_{\bbr^d} \rho^{\eps} u^{\eps} \cdot\nabla_x\phi dxds\rightarrow \int_0^t \int_{\bbr^d} \tilde{\rho} u\cdot\nabla_x\phi dxds \quad\mbox{as}~ \eps\rightarrow 0.
\eeq
It obviously follows from the assumption $(\mathcal{A}2)$ that
\beq\label{claim-3}
\int_{\bbr^d}\rho^{\eps}_0 \phi(0,x)dx \rightarrow \int_{\bbr^d}\rho_0 \phi(0,x)dx \quad\mbox{as}~ \eps\rightarrow 0.
\eeq
Therefore, by \eqref{eps-eq}, \eqref{claim-1}, \eqref{claim-2} and \eqref{claim-3}, we have shown the equation \eqref{claim-11}.\\
On the other hand, since $(\rho, u)$ is a classical solution to the continuity equation 
\[
\partial_t\rho + \nabla_x\cdot(\rho u)=0,
\] 
we can find
\beq\label{classic-eq}
\int_0^t \int_{\bbr^d}(\rho\partial_t \phi + \rho u \cdot \nabla_x\phi)dxds + \int_{\bbr^d}\rho_0 \phi(0,x)dx=0.
\eeq
We now use \eqref{claim-11} and \eqref{classic-eq} to complete the claim \eqref{ae}. \\
For any function $g\in C_c^{\infty}([0,T_*]\times\bbr^d)$, there is a smooth solution 
$\phi\in C_c^{\infty}([0,T_*]\times\bbr^d)$ of
\[
\partial_t\phi +u\cdot\nabla_x\phi = g,
\]
thanks to the smoothness of $u$.\\
Then, by \eqref{claim-11} and \eqref{classic-eq}, we have
\begin{align*}
\begin{aligned} 
\int_0^t \int_{\bbr^d} (\tilde{\rho} -\rho)g dxds &= \int_0^t \int_{\bbr^d} (\tilde{\rho} -\rho)(\partial_t\phi +u\cdot\nabla_x\phi) dxds \\
&=0,
\end{aligned}
\end{align*}
which implies the claim \eqref{ae}.\\
Therefore, by \eqref{weak conv}, \eqref{ae} and \eqref{1flux}, we have shown
\begin{align}
\begin{aligned}\label{mass-mom}
&\rho^{\eps}\rightharpoonup \rho\quad\mbox{weak-}* \quad \mbox{in} ~\mathcal{M}([0,T_*]\times\bbr^d),\\
&\rho^{\eps}u^{\eps} \rightharpoonup {\rho}u\quad\mbox{weak-}* \quad \mbox{in} ~\mathcal{M}([0,T_*]\times\bbr^d).
\end{aligned}
\end{align}

\subsubsection{Convergence of $\rho^{\eps} u^{\eps}\otimes u^{\eps}$}
For any vector-valued function $\Phi\in C_c^{\infty}([0,T_*]\times\bbr^d)$, we consider
\begin{align}
\begin{aligned} \label{2flux}
&\int_0^t \int_{\bbr^d} (\rho^{\eps} u^{\eps} \otimes u^{\eps} - \rho u\otimes u):\Phi dxds\\
&\qquad =\int_0^t \int_{\bbr^d} (\rho^{\eps} (u^{\eps}-u) \otimes u^{\eps} ) :\Phi dxds +\int_0^t \int_{\bbr^d} (\rho^{\eps} u\otimes (u^{\eps}-u) ):\Phi dxd\\
&\hspace{2cm} +\int_0^t \int_{\bbr^d} ((\rho^{\eps} - \rho ) u\otimes u):\Phi dxd\\
&\qquad := \mathcal{J}_1 +\mathcal{J}_2 +\mathcal{J}_3.
\end{aligned}
\end{align}
For $\mathcal{J}_1$, we use the H$\ddot{\mbox{o}}$lder's inequality and minimization property \eqref{mini} to get
\begin{align*}
\begin{aligned} 
\mathcal{J}_1 &\le \|\Phi\|_{L^{\infty}}\int_0^t \Big( \int_{\bbr^d} \rho^{\eps} |u^{\eps}-u|^2 dxds \Big)^{1/2} 
\Big( \int_{\bbr^d} \rho^{\eps} |u^{\eps}|^2 dx  \Big)^{1/2} ds\\
&\le C \int_0^t \Big(\int_{\bbr^d} \rho^{\eps} |u^{\eps}-u|^2 dxds \Big)^{1/2} \Big(\int_{\bbr^d} \mathcal{F} (f^{\eps})(t) dx \Big)^{1/2} ds.
\end{aligned}
\end{align*}
Then, by \eqref{kinetic-ineq}, \eqref{main-ineq} and the assumption $(\mathcal{A}1)$, we have
\begin{align*}
\begin{aligned} 
\mathcal{J}_1 
&\le C T_*  \Big(\int_{\bbr^d} \rho^{\eps} |u^{\eps}-u|^2 dx\Big)^{1/2}  \Big(\int_{\bbr^d} \mathcal{F} (f^{\eps}_0) dx \Big)^{1/2} \rightarrow 0 \quad\mbox{as}~ \eps\rightarrow 0.
\end{aligned}
\end{align*}
For $\mathcal{J}_2$, we use the H$\ddot{\mbox{o}}$lder's inequality and the conservation of mass to estimate
\begin{align*}
\begin{aligned} 
\mathcal{J}_2 &\le \|\Phi \|_{L^{\infty}}\|u\|_{L^{\infty}}\int_0^t \Big( \int_{\bbr^d} \rho^{\eps} dx\Big)^{1/2} \Big(\int_{\bbr^d}  \rho^{\eps} |u^{\eps} -u|^2 dx\Big)^{1/2}ds\\
&\le CT_* \|\rho^{\eps}_0\|_{L^1(\bbr^d)}^{1/2}\Big(\int_{\bbr^d} \rho^{\eps} |u^{\eps} -u|^2 dx\Big)^{1/2} \rightarrow 0 \quad\mbox{as}~ \eps\rightarrow 0.
\end{aligned}
\end{align*}
Taking $u\otimes u :\nabla_x\Phi$ as a test function due to regularity of $u$, it follows from the weak convergence of mass $\eqref{mass-mom}_1$ that 
\[
\mathcal{J}_3\rightarrow 0 \quad\mbox{as}~ \eps\rightarrow 0.
\]
Thus we have
\[
\int_0^t \int_{\bbr^d} (\rho^{\eps} u^{\eps} \otimes u^{\eps} - \rho u\otimes u):\Phi dxds\rightarrow 0 \quad\mbox{as}~ \eps\rightarrow 0.
\]
Therefore, we have shown 
\begin{align*}
\begin{aligned}
\rho^{\eps}u^{\eps}\otimes u^{\eps} \rightharpoonup {\rho}u\otimes u \quad\mbox{weak-}* \quad \mbox{in} ~\mathcal{M}([0,T_*]\times\bbr^d).
\end{aligned}
\end{align*}
\subsubsection{Convergence of $\int_{\bbr^d}f^{\eps} |v|^2 dv$}
For any $\phi\in C^{\infty}_c([0,T_*)\times\bbr^d)$, we have
\begin{align}
\begin{aligned}\label{e-conv} 
&\int_0^t \int_{\bbr^d} \Big(\int_{\bbr^d} f^{\eps}|v|^2 dv - \rho |u|^2 \Big) \phi dx ds\\
&\quad = \int_0^t \int_{\bbr^{2d}} f^{\eps}|v-u^{\eps}|^2 \phi dvdxds + \int_0^t \int_{\bbr^d} (\rho^{\eps} |u^{\eps}|^2 -\rho |u|^2) \phi dxdt\\
&\quad =I_1 +I_2. 
\end{aligned}
\end{align}
By \eqref{diss-zero}, we have 
\begin{align*}
\begin{aligned} 
{I}_1 \le \|\phi\|_{L^{\infty}} \int_0^{t} \mathcal{D}(f^{\eps}) ds  \rightarrow 0\quad\mbox{as}~ \eps\rightarrow 0.
\end{aligned}
\end{align*}
We use \eqref{main-ineq} and \eqref{mass-mom} to get
\begin{align*}
\begin{aligned} 
{I}_2 & = \int_0^t \int_{\bbr^d} \rho^{\eps} |u^{\eps} - u|^2 \phi dxdt + 2\int_0^t \int_{\bbr^d}(\rho^{\eps} u^{\eps} -\rho u)\cdot u\phi dxdt\\
&\quad -\int_0^t \int_{\bbr^d} (\rho^{\eps} - \rho) |u|^2 \phi dxdt\\
&\rightarrow 0.
\end{aligned}
\end{align*}
Therefore, we have 
\[
\int_{\bbr^d}f^{\eps} |v|^2 dv \rightharpoonup \rho |u|^2 \quad\mbox{weak-}* \quad \mbox{in} ~\mathcal{M}([0,T_*]\times\bbr^d).
\]
This complete the proof of theorem.\\
\subsubsection{Convergence of $\int_{\bbr^d}f^{\eps} \psi(v) dv$ in \eqref{dirac-con}}
We first use \eqref{main-ineq} and \eqref{diss-zero} to estimate
\begin{align}
\begin{aligned}\label{new-1} 
\int_0^t \int_{\bbr^{2d}} f^{\eps}|v-u|^2 dvdxds &\le 2\int_0^t \int_{\bbr^{2d}} f^{\eps}(|v-u^{\eps}|^2 + |u^{\eps}-u|^2) dvdxds\\
& \le C(\eps + T^*\sqrt{\eps}) \le C\sqrt{\eps}.
\end{aligned}
\end{align}
For any $\phi\in C^{\infty}_c([0,T_*)\times\bbr^d)$ and $\psi\in C^1(\bbr^d)$ with $\nabla\psi\in L^{\infty}(\bbr^d)$, by using \eqref{new-1} and mean-value theorem, we have
\begin{align*}
\begin{aligned} 
&\int_0^t\int_{\bbr^{2d}} \phi(t,x) f^{\eps} (\psi(v)-\psi(u)) dvdxds \\
&\quad = \int_{|v-u|\le \eps^{1/4}} \phi(t,x) f^{\eps} (\psi(v)-\psi(u)) dvdxds+ \int_{|v-u|>\eps^{1/4}} \phi(t,x) f^{\eps} (\psi(v)-\psi(u)) dvdxds\\
&\quad \le \|\phi\|_{\infty} \|\nabla\psi\|_{\infty} \Big(\int_{|v-u|\le \eps^{1/4}} f^{\eps} |v-u| dvdxds
+ \int_{|v-u|>\eps^{1/4}}f^{\eps} |v-u| dvdxds \Big)\\
&\quad \le \|\phi\|_{\infty} \|\nabla\psi\|_{\infty} \Big(\eps^{1/4} T^*\|f^{\eps}\|_{L^1(\bbr^{2d})}
+ \eps^{-1/4}\int_0^t \int_{\bbr^{2d}} f^{\eps} |v-u|^2 dvdxds \Big)\\
&\quad \le C \eps^{1/4},
\end{aligned}
\end{align*} 
which yields
\[
\int_{\bbr^d}f^{\eps} \psi(v) dv \rightharpoonup \rho^{\eps} \psi(u) \quad\mbox{weakly} \quad \mbox{in} ~\mathcal{M}([0,T_*]\times\bbr^d).
\]
Therefore, by the convergence of mass $\eqref{mass-mom}_1$, we complete
\[
\int_{\bbr^d}f^{\eps} \psi(v) dv \rightharpoonup \rho \psi(u) \quad\mbox{weakly} \quad \mbox{in} ~\mathcal{M}([0,T_*]\times\bbr^d).
\]

\begin{appendix}
\section{Proof of Proposition \ref{prop-key}}
First of all, by the definition of the relative entropy $\eqref{relative}_1$, we have 
\begin{align*}
\begin{aligned}
\frac{d}{dt}\int_{\bbr^d}\mathcal{E}(V|U) dx & = \int_{\bbr^d}\partial_t\mathcal{E}(V) dx-\int_{\bbr^d} D\mathcal{E}(U)\cdot\partial_t U dx\\
&\quad -\int_{\bbr^d} D^2\mathcal{E}(U) \partial_t U \cdot (V-U) dx - \int_{\bbr^d} D\mathcal{E}(U) \partial_t (V-U) dx \\
&=\int_{\bbr^d}\partial_t\mathcal{E}(V) dx-\int_{\bbr^d} D^2\mathcal{E}(U) \partial_t U \cdot (V-U) dx - \int_{\bbr^d} D\mathcal{E}(U)\cdot \partial_t V dx \\
&=:\sum_{l=1}^3\mathcal{I}_K.
\end{aligned}
\end{align*}
By \eqref{system}, we rewrite $\mathcal{I}_2$ as
\begin{align*}
\begin{aligned}
\mathcal{I}_2 &= \int_{\bbr^d} D^2\mathcal{E}(U) (\mbox{div}_x A(U) -\lambda F(U)) \cdot (V-U) dx\\
&=  \int_{\bbr^d} \nabla_xD\mathcal{E}(U): DA(U) \cdot (V-U)dx - \lambda\int_{\bbr^d} D^2\mathcal{E}(U) F(U) \cdot (V-U) dx,
\end{aligned}
\end{align*}
where we have used the following formula on integration by parts: (See \cite{B-V,K-M-T-2} for its derivation.)
\[
 \int_{\bbr^d} D^2\mathcal{E}(U) \mbox{div}_x A(U) \cdot (V-U)dx =  \int_{\bbr^d} \nabla_xD\mathcal{E}(U): DA(U) \cdot (V-U)dx .
\]
By adding and subtracting, $\mathcal{I}_3$ can be written as
\begin{align*}
\begin{aligned}
\mathcal{I}_3 &= - \int_{\bbr^d} D\mathcal{E}(U) \cdot(\partial_t V + \mbox{div}_x A(V)-\lambda F(V))dx\\
&\quad +\int_{\bbr^d} D\mathcal{E}(U) \cdot\mbox{div}_x A(V)dx -\lambda \int_{\bbr^d} D\mathcal{E}(U)\cdot F(V) dx\\
&= - \int_{\bbr^d} D\mathcal{E}(U)\cdot (\partial_t V + \mbox{div}_x A(V)-\lambda F(V))dx\\
&\quad -\int_{\bbr^d} \nabla_x D\mathcal{E}(U) : A(V)dx -\lambda \int_{\bbr^d} D\mathcal{E}(U)\cdot F(V) dx
\end{aligned}
\end{align*}
We use the relative flux $\eqref{relative}_2$ to get
\begin{align*}
\begin{aligned}
\mathcal{I}_2+\mathcal{I}_3 &=  -\int_{\bbr^d} \nabla_x D\mathcal{E}(U) : A(V|U)dx-\int_{\bbr^d} D\mathcal{E}(U) \cdot[\partial_tV + \mbox{div}_x A(V) -\lambda F(V)] dx\\
&\quad -\lambda \int_{\bbr^d} (D^2\mathcal{E}(U) F(U)\cdot (V-U) +  D\mathcal{E}(U) \cdot F(V) ) dx
 -\int_{\bbr^d} \nabla_x D\mathcal{E}(U) : A(U).
\end{aligned}
\end{align*}
Thanks to \eqref{entropy-flux}, the last term vanish as follows.
\begin{align*}
\begin{aligned}
 -\int_{\bbr^d} \nabla_x D\mathcal{E}(U) : A(U) = \int_{\bbr^d} D\mathcal{E}(U) \cdot\mbox{div}_x A(U)dx=\int_{\bbr^d} \mbox{div}_x G(U)dx =0,
\end{aligned}
\end{align*}
which complete the proof.
\end{appendix}


\begin{thebibliography}{99}


\bibitem{B-G-L-1} C. Bardos, F. Golse, and C. D. Levermore, : \textit{Fluid dynamic limits of kinetic equations. I. Formal derivations.} J. Statist. Phys., 63(1-2), (1991) pp.323--344.
\bibitem{B-G-L-2} C. Bardos, F. Golse, and C. D. Levermore, : \textit{Fluid dynamic limits of kinetic equations. II. Convergence proofs for the Boltzmann equation}. Comm. Pure Appl. Math., 46(5), (1993) 667--753.


\bibitem{B-T-V} F. Berthelin, A. E. Tzavaras, and A. Vasseur, : \textit{From discrete velocity Boltzmann equations to gas dynamics before shocks}. J. Stat. Phys., 135(1), (2009) 153--173.

\bibitem{B-V} F. Berthelin and A. Vasseur,: \textit{From kinetic equations to multidimensional isentropic gas dynamics before shocks.} SIAM J. Math. Anal. 36(6), 1807--1835 (2005)

\bibitem{B-G} Y. Brenier and E. \textit{Grenier, Sticky particles and scalar conservation laws}, SIAM J. Numer. Anal., 35, 2317--2328 (1998).





\bibitem{Bo} L. Boudin : \textit{A solution with bounded expansion rate to the model of viscous pressureless gases}. SIAM J. Math. Anal. 32, No. 1, 172--193 (2000).

\bibitem{C-C-R} J. A. C$\tilde{\mbox{a}}$nizo, J. A. Carrillo and J. Rosado, \textit{A well-posedness theory in measures for some kinetic
models of collective motion}, Math. Models Methods Appl. Sci., 21, 515--539 (2011).

\bibitem{C-C-K} J. A. Carrillo, Y.-P. Choi, T. K. Karper, \textit{On the analysis of a coupled kinetic-fluid model with local alignment forces,} preprint (2013).

\bibitem{C-F-R-T} J. A. Carrillo, M. Fornasier, J. Rosado and G. Toscani, \textit{Asymptotic flocking dynamics for the kinetic Cucker-Smale model.} SIAM J. Math. Anal. 42, 218--236 (2010).

\bibitem{C-L}  G.-Q. Chen and H. L. Liu,   \textit{Formation of $\delta$-shocks and vacuum states in the vanishing pressure limit of solutions to the Euler equations for isentropic fluids}, SIAM J. Math Anal. 34, 925--938 (2003).

\bibitem{C-V} K. Choi and A. Vasseur, : \textit{Relative entropy applied to the stability of viscous shocks up to a translation for scalar conservation laws}, Submitted.

\bibitem{C-K-F-L} I.D. Couzin, J. Krause, N.R. Franks, and S.A. Levin, \textit{Effective leadership and decision-making in animal groups on the move.} Nature, 433(7025):513--516, (2005).

\bibitem{C-S} F. Cucker and S. Smale, : \textit{Emergent behavior in flocks}. IEEE Trans. Automat. Control 52, 852--862 (2007).

\bibitem{D} C.M. Dafermos, \textit{The second law of thermodynamics and stability}. Arch. Ration. Mech. Anal. 70 (2), 167--179 (1979).

137Ð188.


\bibitem{G-S} F. Golse and L. Saint-Raymond, : \textit{The Navier-Stokes limit of the Boltzmann equation for bounded collision kernels}. Invent. Math., 155(1), (2004) 81--161.


\bibitem{G-J-V-2} T. Goudon, P.-E Jabin, A. Vasseur, : \textit{Hydrodynamic limit for the Vlasov-Navier-Stokes equations. II.
Fine particles regime}. Indiana Univ. Math. J. 53(6), 1517Ð1536 (2004)


\bibitem{H-K-K1} S.-Y. Ha, M.-J. Kang and B. Kwon, \textit{A hydrodynamic model for the interaction of Cucker-Smale particles and incompressible fluid}, Math. Models Methods Appl. Sci., 11, 2311--2359 (2014).

\bibitem{H-K-K2} S.-Y. Ha, M.-J. Kang and B. Kwon, \textit{Emergent dynamics for the hydrodynamic Cucker-Smale system in a moving domain}, Preprint (2014).

\bibitem{H-L} S.-Y. Ha and J. G. Liu, \textit{A simple proof of the Cucker-Smale flocking dynamics and mean-field limit}. Communications in Mathematical Sciences, 7(2), 297--325, (2009).


\bibitem{H-T} S. Y. Ha and E. Tadmor, \textit{From particle to kinetic and hydrodynamic descriptions of flocking}. Kinetic and Related Models, 1(3), 415--435, (2008).

\bibitem{H-W} F. Huang and Z. Wang, \textit{Well posedness for pressureless flow}, Comm. Math. Phys., 222, 117--146 (2001).


\bibitem{J} P.-E. Jabin :\textit{Macroscopic limit of Vlasov type equations with friction}. Ann. Inst. H. Poincare Anal. Non
Lineaire 17, 651-672 (2000).

\bibitem{K-M-T-1} T. K. Karper, A. Mellet and K. Trivisa, \textit{Existence of weak solutions to kinetic flocking models}. SIAM J. Math. Anal., 45(1), 215 -- 243 (2013).

\bibitem{K-M-T-2} T. K. Karper, A. Mellet and K. Trivisa, \textit{Hydrodynamic limit of the kinetic Cucker-Smale flocking model}. to appear in Math. Models Methods Appl. Sci.

\bibitem{K-M-T-3} T. K. Karper, A. Mellet and K. Trivisa, \textit{On strong local alignment in the kinetic Cucker-Smale model}, Hyperbolic Conservation Laws and Related Analysis with Applications Springer Proceedings in Math. Statistics Vol. 49, 227--242 (2014).

\bibitem{K-V} M.-J. Kang and A. Vasseur : \textit{$L^2$-contraction for shock waves of scalar viscous conservation laws}. preprint (2014).

\bibitem{L} N. Leger : \textit{$L^2$ stability estimates for shock solutions of scalar conservation laws using the relative entropy method}, Arch. Ration. Mech. Anal., 199(3), 761--778, (2011).

\bibitem{L-V} N. Leger and A. Vasseur : \textit{Relative entropy and the stability of shocks and contact discontinuities for systems of conservation laws with non-BV perturbations}, Arch. Ration. Mech. Anal., 201(1), 271--302, (2011).

\bibitem{L-M} P.-L. Lions and N. Masmoudi : \textit{From the Boltzmann equations to the equations of incompressible fluid mechanics. I, II.} Arch. Ration. Mech. Anal., 158(3), (2001) 173--193, 195--211.

\bibitem{M-S} N. Masmoudi and L. Saint-Raymond : \textit{From the Boltzmann equation to the Stokes-Fourier system in a bounded domain,} Comm. Pure Appl. Math., 56(9):(2003)1263--1293.

\bibitem{M-V} A. Mellet and A. Vasseur, \textit{Asymptotic Analysis for a Vlasov-Fokker-Planck Compressible
Navier-Stokes Systems of Equations}. Commun. Math. Phys. 281, 573--596 (2008).

\bibitem{M-T-1} S. Motsch and E. Tadmor. \textit{A new model for self-organized dynamics and its flocking behavior.} Journal of Statistical Physics, Springer, 141 (5): 923--947, (2011).

\bibitem{M-T-2} S. Motsch and E. Tadmor. \textit{Heterophilious dynamics enhances consensus}, preprint (2013)

\bibitem{S} L. Saint-Raymond, \textit{Convergence of solutions to the Boltzmann equation in the incompressible Euler limit}. Arch. Ration. Mech. Anal., 166(1), (2003) 47--80.

\bibitem{Sh} J. Shen, \textit{Cucker-Smale flocking under hierarchical leadership.} SIAM J. Appl. Math. 68(3), 694--719 (2008).

\bibitem{S-S-Z} J. Silk, A. Szalay and Ya. B. Zeldovich : \textit{Large-scale structure of the universe.} Scientific American, 249, 72--80 (1983).

\bibitem{S-V} D. Serre and A. Vasseur : \textit{$L^2$-type contraction for systems of conservation laws}, Preprint.

\bibitem{V} A. Vasseur : \textit{Relative entropy and contraction for extremal shocks of Conservation Laws up to a shift}, Preprint, 2013.

\bibitem{V-1} A. Vasseur, \textit{Recent results on hydrodynamic limits.} In Handbook of differential equations:
evolutionary equations. Vol. IV, Handb. Differ. Equ., 323--376. Elsevier/North-Holland, Amsterdam, (2008).


\bibitem{Y} H.-T. Yau : \textit{Relative entropy and hydrodynamics of Ginzburg-Landau models.} Lett. Math. Phys., 22(1), (1991) 63--80.

\bibitem{W-R-S} E. Weinan, Yu. G. Rykov and Ya. G. Sinai, \textit{Generalized variational principles, global weak solutions and behavior with random initial data for systems of conservation laws arising in adhesion particle dynamics}, Commun. Math. Phys., 177, 349--380 (1996).

\bibitem{Z-E-P}  M. Zavlanos, M. Egerstedt, and G. J. Pappas, \textit{Graph theoretic connectivity control of mobile robot networks}. Proceedings of the IEEE, 99(9):1525--1540, 2011.

\bibitem{Ze} Ya. B. Zeldovich: \textit{Gravitational instability: An approximate theory for large density perturbations}. Astro. Astrophys., 5, 84 (1970).




\end{thebibliography}
\end{document}